\documentclass{article}

\usepackage{geometry}
\usepackage{graphicx} 
\usepackage{epstopdf} 

\usepackage[colorlinks, hypertexnames=false]{hyperref}
\hypersetup{linkcolor=blue, citecolor=red} 

\usepackage{bm}
\usepackage{amsmath} 
\usepackage{amssymb} 
\usepackage{stmaryrd} 
\usepackage{multirow}

\usepackage{amsthm} 
\usepackage{mathtools} 

\usepackage[dvipsnames]{xcolor}
\definecolor{red}{rgb}{1.00,0.00,0.00}
{\numberwithin{equation}{section}
\setlength{\parindent}{1em}

\newtheorem{theorem}{Theorem}[section]
\newtheorem{lemma}{Lemma}[section]

\renewcommand{\footnotesize}{\scriptsize}
\newcommand{\normmm}[1]{{\left\vert\kern-0.25ex\left\vert
\kern-0.25ex\left\vert #1
    \right\vert\kern-0.25ex\right\vert\kern-0.25ex\right\vert}}

\geometry{left=3cm,right=3cm,top=4cm,bottom=2.5cm}
\newcommand{\bs}{\boldsymbol}
\newcommand{\matr}[1]{#1}

\begin{document}           

\title{A uniformly robust staggered DG method for the unsteady Darcy-Forchheimer-Brinkman problem}
\author{Lina Zhao\footnotemark[1]\qquad
\;Ming Fai Lam\footnotemark[2]\qquad
\;Eric Chung\footnotemark[3]}
\renewcommand{\thefootnote}{\fnsymbol{footnote}}
\footnotetext[1]{Department of Mathematics, The Chinese University of Hong Kong, Hong Kong SAR, China. ({lzhao@math.cuhk.edu.hk}).}
\footnotetext[2]{Department of Mathematics, The Chinese University of Hong Kong, Hong Kong SAR, China. ({mflam@math.cuhk.edu.hk}).}
\footnotetext[3]{Department of Mathematics, The Chinese University of Hong Kong, Hong Kong SAR, China. ({tschung@math.cuhk.edu.hk}).}

\maketitle

\textbf{Abstract:} In this paper we propose and analyze a uniformly robust staggered DG method
for the unsteady Darcy-Forchheimer-Brinkman problem. Our formulation is based on
velocity gradient-velocity-pressure and the resulting scheme can be flexibly
applied to fairly general
polygonal meshes. We relax the tangential continuity for velocity,
 which is the key ingredient in achieving the uniform robustness. We present well-posedness
and error analysis for both the semi-discrete scheme and the fully discrete scheme,
 and the theories indicate that the error estimates for velocity are independent of pressure.
 Several numerical experiments are presented to confirm the theoretical findings.


\pagestyle{myheadings} \thispagestyle{plain}
\markboth{ZhaoLamChung} {SDG method for Darcy-Forchheimer-Brinkman problem}

\section{Introduction}

The Brinkman problem can model fluid motion in porous media with fractures. A distinctive
 feature of Brinkman model is that it can behave like Stokes or Darcy problem by tuning
 the parameters related to the fluid viscosity and the medium permeability, respectively.
 But this at the same time brings additional technical difficulties for the design of
 numerical schemes that are robust in both Stokes and Darcy regimes. In general,
 the traditional Stokes stable elements will suffer from loss of convergence when
  the Brinkman problem becomes Darcy dominating, vice versa. To overcome this issue,
  numerous numerical schemes have been developed for the Birnkman problem
 \cite{Konno11,MardalTaiWinther02,BurmanHansbo05,BurmanHansbo07,Badia09,Guzman12}.
Darcy's law is widely employed to model flow in porous media, becomes
unreliable for Reynolds numbers greater than one. The Forchheimer model \cite{Forchheimer1901}
accounts for faster flows by including a nonlinear inertial term and has
been extensively studied \cite{KimPark99,Park05,Girault08,PanRui12,RuiPan12,RuiLiu15}.
The Brinkman-Forchheimer model \cite{Payer99,Celebi06,Louaked17} combines the advantages of the two models and can
be used for fast flows in highly porous media.

Staggered discontinuous Galerkin method as a new generation of numerical schemes is initially proposed by Chung and Engquist
to solve wave propagation problems \cite{EricEngquistwave,ChungWave2}, and
inspires study for various partial differential equations arising from practical
applications \cite{ChungCiarletYu13,KimChung13,KimChungLam16,ChungQiu17,ChungParkZhao18,DuChung18}.
Recently, the concept of staggered DG method is integrated into polygonal methods
and have been successfully designed for numerous mathematical models that have important
 physical applications \cite{LinaPark18,LinaParkShin19,LinaParkSD20,LinaChungLam20,LinaChungParkZhou19,
 KimZhaoPark20,LinaParkElasticity20,LinaPark20}. A uniformly stable staggered DG method has been established
 for the Brinkman problem by relaxing the tangential continuity of velocity in \cite{LinaChungLam20}.
The numerical experiments presented therein verify that the proposed scheme is
robust with respect to viscosity and the accuracy of velocity remains almost the same
 for various values of viscosity. This feature is desirable in practical applications,
 therefore, the purpose of this paper is to extend the staggered DG method developed
 in \cite{LinaChungLam20} to the unsteady Darcy-Forchheimer-Brinkman problem.

Our formulation is based on velocity gradient-velocity-pressure,
and the finite element spaces for these three variables enjoy staggered continuity properties.
Specifically, the finite element space for velocity is continuous in
the normal direction over the dual edges and the finite element space for pressure is
continuous over the primal edges. Thus, our approach can be viewed as Darcy tailored method.
We also emphasize that relaxing tangential continuity for velocity is the crux in the design of
our uniformly stable scheme and the spaces exploited in \cite{LinaPark18,LinaPark20} will lead to loss of convergence when
viscosity approaches zero.
There are several desirable features of this approach, which can be summarized as follows:
First, our method can be flexibly applied to fairly general meshes possibly
including hanging nodes. Second, all the variables of physical interest can be
calculated simultaneously. Third, it is uniformly robust with respect to
the Brinkman parameter in both the Stokes and Darcy regimes,
and the accuracy of velocity remains almost the same with various values of the Brinkman parameter.
In this paper, we analyze the convergence for both the semi-discrete scheme
and the fully discrete scheme, where backward Euler is employed for the time discretization.
We can achieve optimal rates of convergence, and the convergence of velocity
is independent of the Brinkman parameter. Finally, we perform several numerical experiments
to verify the proposed theories and we can observe that our method is
uniformly robust with respect to the parameters, as expected, the accuracy of velocity
 remains almost the same for various values of the parameters.

The rest of the paper is organized as follows. In the next section,
we present the continuous weak formulation, in addition, the unique solvability and
stability is proved. In section~\ref{sec:sdg}, we describe the discrete formulation for the
unsteady Darcy-Forchheimer-Brinkman problem. Then unique solvability of the discrete
formulation and the error estimates are presented in section~\ref{sec:analysis}.
Several numerical experiments are carried out in section~\ref{sec:numerical} to illustrate
that the proposed method is uniformly robust with respect to the parameters.
Finally, a conclusion is given.

\section{The continuous formulation}

We consider the following unsteady Darcy-Forchheimer-Brinkman model:
\begin{equation}\label{sys0}
\begin{alignedat}{2}
  \bm{u}_t-\epsilon\Delta \bm{u}+\alpha \bm{u}+\beta|\bm{u}|\bm{u}+\nabla p &= \bm{f}\quad&&{\rm in~}\Omega, t\in[0,T],\\
  \text{div}\;\bm{u}&=0\quad&&{\rm in~}\Omega, t\in[0,T],\\
  \bm{u}&=\bm{0} \quad&&{\rm on~}\partial\Omega, t\in[0,T],\\
  \bm{u}(\bm{x},0)&=\bm{u}_0 \quad&&{\rm in~}\Omega,
\end{alignedat}
\end{equation}
where $\alpha>0$ is the Darcy coefficient,
$\epsilon>0$ is the Brinkman coefficient, $\beta>0$ is the Forchheimer coefficient
 and $\bm{f}\in L^2(\Omega)^2$ is
the external body force. In addition, we also assume
that there exist real numbers $\alpha_{\text{min}}$,
$\alpha_{\text{max}}$ and $\beta_{\text{max}}$ such
that $0<\alpha_{\text{min}}\leq \alpha\leq \alpha_{\text{max}}$ and $\beta\leq \beta_{\text{max}}$.
The unknowns are the velocity $\bm{u}$ and the pressure $p$, which are functions
of the spatial variable $\bm{x}$ and temporal variable $t$ in $[0,T]$ ($T>0$ is a finite time).
Here $|\bm{u}|=\sqrt{u_1^2+u^2_2+\cdots+u_n^2}$ for $\bm{u}\in\mathbb{R}^n$.

We introduce an additional unknown $L=\sqrt{\epsilon}\bm{u}$, then the above system of equations can be recast into the following first order system of equations:
\begin{equation}\label{sys1}
\begin{alignedat}{2}
  L-\sqrt{\epsilon}\nabla \bm{u}&=0\quad&&{\rm in~}\Omega, t\in[0,T],\\
  \bm{u}_t-\sqrt{\epsilon}\text{div}\;L+\alpha \bm{u}+\beta|\bm{u}|\bm{u}+\nabla p &= \bm{f}\quad&&{\rm in~}\Omega, t\in[0,T],\\
  \text{div}\;\bm{u}&=0\quad&&{\rm in~}\Omega, t\in[0,T],\\
  \bm{u}&=\bm{0} \quad&&{\rm on~}\partial\Omega, t\in[0,T],\\
  \bm{u}(\bm{x},0)&=\bm{u}_0 \quad&&{\rm in~}\Omega.
\end{alignedat}
\end{equation}

%
Here for simplicity we assume $\bm{u}_0=\bm{0}$ for the subsequent analysis. We introduce some notations that will be used later. For a set $D\subset \mathbb{R}^{2}$, we denote the scalar product in $L^{2}(D)$ by $(\cdot,\cdot)_{D}$, namely $(p,q)_{D}:=\int_{D} p\,q \;dx$, we use the same symbol $(\cdot,\cdot)_{D}$ for the inner product in $L^{2}(D)^{2}$ and in $L^{2}(D)^{2\times 2}$. More precisely $(\sigma,\tau)_{D}:
=\sum_{i=1}^{2}\sum_{j=1}^{2}(\sigma^{ij},
\tau^{ij})_{D}$ for $\sigma,\tau\in L^2(D)^{2\times 2}$. When $D$
coincides with $\Omega$, the subscript $\Omega$ will be dropped
unless otherwise mentioned.
We denote by $(\cdot, \cdot)_e$ the scalar product in
$L^2(e), e\subset \mathbb{R}$ (or duality pairing), for a scalar, vector, or tensor functions.
Given an integer  $m\geq 0$ and $p\geq 1$, $W^{m,p}(D)$ and
$W_0^{m,p}(D)$ denote the usual Sobolev space provided the norm
and semi-norm $\|v\|_{W^{m,p}(D)}=\{\sum_{|\alpha|\leq m}\|D^\alpha
v\|^p_{L^p(D)}\}^{1/p}$, $|v|_{W^{m,p}(D)}=\{\sum_{|\alpha|=
m}\|D^\alpha v\|^p_{L^p(D)}\}^{1/p}$. If $p=2$ we usually write
$H^m(D)=W^{m,2}(D)$ and $H_0^m(D)=W_0^{m,2}(D)$,
$\|v\|_{H^m(D)}=\|v\|_{W^{m,2}(D)}$ and $|v|_{H^m(D)}=|v|_{W^{m,2}(D)}$.
In addition, we need spaces of vector values functions such as
$L^2(0,T;H^m(\Omega))$ and $C(0,T;H^m(\Omega))$ with the norms
\begin{align*}
\|\psi\|_{L^2(0,T;H^m(\Omega))}=\Big(\int_0^{T} \|\psi(t)\|_{H^m(\Omega)}^2\;dt\Big)^{1/2},\quad
\|\psi\|_{C(0,T;H^m(\Omega))}=\max_{0\leq t\leq T}\|\psi(t)\|_{H^m(\Omega)}.
\end{align*}

Integration by parts yields the following weak formulation:
Find $(L,\bm{u},p)\in L^{3/2}(\Omega)^{2\times 2}\times W_0^{1,3}(\Omega)^2\times L^{3/2}(\Omega)$
such that
\begin{equation}\label{eq:continuous}
\begin{alignedat}{2}
  (L,G)-\sqrt{\epsilon}(\nabla \bm{u},G)&=0\quad&&\forall G\in L^{3/2}(\Omega)^{2\times 2},\\
  (\bm{u}_t,\bm{v})+\sqrt{\epsilon}(L,\nabla \bm{v})
  +(\alpha\bm{u},\bm{v})+(\beta|\bm{u}|\bm{u},\bm{v})-(p,\text{div}\; \bm{v}) &=  (\bm{f},\bm{v})
  \quad&&\forall \bm{v}\in W_0^{1,3}(\Omega)^2,\\
  (\text{div}\;\bm{u},q)&=0\quad&&\forall q\in L^{3/2}(\Omega),
\end{alignedat}
\end{equation}
where
\begin{align*}
W^{1,3}_0(\Omega)^2=\{\bm{v}\in W^{1,3}(\Omega)^2,\bm{v}=\bm{0}\;\mbox{on}\;\partial \Omega\}.
\end{align*}
The bilinear forms given above satisfy the following inf-sup condition:
\begin{align}
\inf_{\bm{v}\in W^{1,3}_0(\Omega)^2}\frac{(\nabla \bm{v},G)}{\|G\|_{L^{3/2}(\Omega)}
\|\bm{v}\|_{W^{1,3}(\Omega)}}\geq C
\label{eq:inf-sup1}
\end{align}
and (cf. \cite{Amrouche94})
\begin{align}
\inf_{q\in L^2(\Omega)}\sup_{\bm{v}\in W_0^{1,3}(\Omega)^2}
\frac{(\nabla \cdot \bm{v},q)_{\Omega}}{\|q\|_{L^{3/2}(\Omega)}\|\bm{v}\|_{W^{1,3}(\Omega)}}\geq C.\label{eq:in-supb}
\end{align}

To ease later analysis, we define for any $\bm{v}\in L^3(\Omega)^2$
\begin{align}
\mathcal{N}(\bm{v})=\alpha\bm{v}+\beta|\bm{v}|\bm{v}.\label{eq:N}
\end{align}
We describe some properties for $\mathcal{N}$, which will play an important role for later analysis.
One can refer to \cite{Girault08} for more details. First, we have
\begin{align}
|\mathcal{N}(\bm{v})-\mathcal{N}(\bm{w})|\leq \alpha_{\text{max}} |\bm{v}-\bm{w}|+
\beta_{\text{max}}|\bm{v}-\bm{w}|(|\bm{v}|+|\bm{w}|), \quad \forall \bm{v},\bm{w}\in L^3(\Omega)^2.\label{eq:property1}
\end{align}
%
\begin{lemma}\label{lemma:monotone}
For fixed $\bm{u}_\ell\in L^3(\Omega)^2$, the mapping
$\bm{u}\rightarrow \mathcal{N}(\bm{u}+\bm{u}_\ell)$ defined by \eqref{eq:N} is monotone
 from $L^3(\Omega)^2$ into $L^{3/2}(\Omega)^2$:
\begin{align}
\forall \bm{u},\bm{v}\in L^3(\Omega)^2, \int_{\Omega}
(\mathcal{N}(\bm{u}+\bm{u}_\ell)-\mathcal{N}(\bm{v}+\bm{u}_\ell))\cdot (\bm{u}-\bm{v})\;dx\geq
 \alpha_{\textnormal{min}}\|\bm{u}-\bm{v}\|_{L^2(\Omega)}^2\label{eq:property2}.
\end{align}

\end{lemma}

\begin{lemma}\label{lemma:coercivity}
For fixed $\bm{u}_l\in L^3(\Omega)^2$, the mapping $\bm{u}\rightarrow \mathcal{N}(\bm{u}+\bm{u}_l)$
 defined by \eqref{eq:N} is coercive in $L^3(\Omega)^2$
\begin{align*}
\lim_{\|\bm{u}\|_{L^3(\Omega)}\rightarrow \infty}
\Big(\frac{1}{\|\bm{u}\|_{L^3(\Omega)}}\int_{\Omega}
\mathcal{N}(\bm{u}+\bm{u}_l)\cdot\bm{u}\;dx\Big)
=\infty.
\end{align*}

\end{lemma}

\begin{lemma}\label{lemma:hemicontinuity}
The mapping $\mathcal{N}$ is hemi-continuous in $L^3(\Omega)^2$; for fixed $\bm{u}_l,\bm{u}$
 and $\bm{v}$ in $L^3(\Omega)^2$, the mapping
\begin{align*}
\gamma\rightarrow \int_{\Omega} \mathcal{N}(\bm{u}_l+\bm{u}+\gamma\bm{v})\cdot\bm{v}\;dx
\end{align*}
is continuous from $\mathbb{R}$ into $\mathbb{R}$.

\end{lemma}

%

%
%

\begin{theorem}\label{thm:continuous}
There exists a unique solution to \eqref{eq:continuous}. In addition, the following estimate holds
\begin{align}
\int_0^t\Big(\|L\|_{L^2(\Omega)}^2+
\frac{\alpha_{\textnormal{min}}}{2}\|\bm{u}\|_{L^2(\Omega)}^2\Big)\;ds+\frac{1}{2}\|\bm{u}(t)\|_{L^2(\Omega)}^2
\leq \int_0^t\frac{1}{2\alpha_{\textnormal{min}}}\|\bm{f}\|_{L^2(\Omega)}^2\;ds.\label{eq:stability}
\end{align}

\end{theorem}

\begin{proof}

Since $\mathcal{N}$ is monotone, coercive and hemi-continuous
(cf. Lemmas~\ref{lemma:monotone}-\ref{lemma:hemicontinuity}), in addition, the
inf-sup condition \eqref{eq:inf-sup1} and \eqref{eq:in-supb} hold, we can follow
\cite{Showalter97,Caucao19} to show that there exists a solution
to \eqref{eq:continuous} and we omit the proof for simplicity.
Next, we show that the solution is unique.
Assume that the solution of \eqref{eq:continuous} is not unique.
Let $(L^{i},\bm{u}^i,p^i)$ with $i\in\{1,2\}$ be two solutions corresponding to the same data.
Then, taking \eqref{eq:continuous} with $(G,\bm{v},q)=(L^1-L^2,\bm{u}^1-\bm{u}^2,p^1-p^2)$
and summing up the resulting equations,
we can infer that
\begin{align*}
\|L^1-L^2\|_{L^2(\Omega)}^2+\frac{1}{2}\partial_t\|\bm{u}^1-\bm{u}^2\|_{L^2(\Omega)}^2
+(\mathcal{N}(\bm{u}^1)-\mathcal{N}(\bm{u}^2),\bm{u}^1-\bm{u}^2)=0,
\end{align*}
which yields
\begin{align*}
\|L^1-L^2\|_{L^2(\Omega)}^2+\frac{1}{2}\partial_t\|\bm{u}^1-\bm{u}^2\|_{L^2(\Omega)}^2
+\|\bm{u}^1-\bm{u}^2\|_{L^2(\Omega)}^2\leq 0.
\end{align*}
Integrating in time from $0$ to $t\in [0,T]$ and using $\bm{u}^1(0)=\bm{u}^2(0)$, we obtain
\begin{align*}
\frac{1}{2}\|\bm{u}^1(t)-\bm{u}^2(t)\|_{L^2(\Omega)}^2+\int_0^t \Big(\|L^1-L^2\|_{L^2(\Omega)}^2+
\|\bm{u}^1-\bm{u}^2\|_{L^2(\Omega)}^2\Big) \leq 0.
\end{align*}
Therefore, we can infer that $\bm{u}^1(t)=\bm{u}^2(t)$ and $L^1(t)=L^2(t)$, which
 implies that \eqref{eq:continuous} has a unique solution.

Next, we will show the stability estimate \eqref{eq:stability}.
Taking $G=L$, $\bm{v}=\bm{u}$ and $q=p$ in \eqref{eq:continuous}, we can get
\begin{align*}
\|L\|_{L^2(\Omega)}^2+\frac{1}{2}\frac{d}{dt}\|\bm{u}\|_{L^2(\Omega)}^2+(\mathcal{N}(\bm{u}),\bm{u})=(\bm{f},\bm{u})
\leq \frac{1}{2\alpha_{\textnormal{min}}}\|\bm{f}\|_{L^2(\Omega)}^2
+\frac{\alpha_{\textnormal{min}}}{2}\|\bm{u}\|_{L^2(\Omega)}^2.
\end{align*}
It follows from the definition of $\mathcal{N}$ that
\begin{align*}
(\mathcal{N}(\bm{u}),\bm{u})\geq \alpha_{\text{min}} \|\bm{u}\|_{L^2(\Omega)}^2,
\end{align*}
thereby we can infer that
\begin{align*}
\|L\|_{L^2(\Omega)}^2+\frac{1}{2}\frac{d}{dt}\|\bm{u}\|_{L^2(\Omega)}^2+
\alpha_{\text{min}}\|\bm{u}\|_{L^2(\Omega)}^2
\leq \frac{1}{2\alpha_{\text{min}}}\|\bm{f}\|_{L^2(\Omega)}^2
+\frac{\alpha_{\text{min}}}{2}\|\bm{u}\|_{L^2(\Omega)}^2,
\end{align*}
which yields
\begin{align*}
\|L\|_{L^2(\Omega)}^2+\frac{1}{2}\frac{d}{dt}\|\bm{u}\|_{L^2(\Omega)}^2
+\frac{\alpha_{\text{min}}}{2}\|\bm{u}\|_{L^2(\Omega)}^2
\leq  \frac{1}{2\alpha_{\text{min}}}\|\bm{f}\|_{L^2(\Omega)}^2.
\end{align*}
Integrating over time and using the fact that $\bm{u}(0)=\bm{0}$ imply
\begin{align*}
\int_0^t\Big(\|L\|_{L^2(\Omega)}^2+\frac{\alpha_{\text{min}}}{2}\|\bm{u}\|_{L^2(\Omega)}^2\Big)\;ds
+\frac{1}{2}\|\bm{u}(t)\|_{L^2(\Omega)}^2
\leq \int_0^t\frac{1}{2\alpha_{\text{min}}}\|\bm{f}\|_{L^2(\Omega)}^2\;ds.
\end{align*}
Therefore, the proof is completed.

\end{proof}

\section{Description of staggered DG method}\label{sec:sdg}

In this section, we introduce the discrete formulation for the unsteady
Darcy-Forchheimer-Brinkman problem \eqref{sys0}. To this end, we first introduce the
construction of our staggered DG spaces, in line with this we then present the construction of staggered DG method.
 To begin, we construct three meshes: the primal mesh $\mathcal{T}_{u}$, the dual mesh $\mathcal{T}_{d}$, and the primal simplicial submeshes $\mathcal{T}_h$. For a polygonal domain $\Omega$, consider a general mesh $\mathcal{T}_{u}$ (of $\Omega$) that consists of nonempty connected close disjoint subsets of $\Omega$:
\begin{align*}
\bar{\Omega}=\bigcup_{E\in \mathcal{T}_{u}}E.
\end{align*}
We let $\mathcal{F}_{u}$ be the set of all primal edges in this partition and $\mathcal{F}_{u}^{0}$ be the
subset of all interior edges, that is, the set of edges in $\mathcal{F}_{u}$ that do not lie on $\partial\Omega$. We construct the primal submeshes $\mathcal{T}_h$ as a triangular subgrid of the primal grid: for an element $E\in \mathcal{T}_{u}$, elements of $\mathcal{T}_h$ are obtained by connecting the interior point $\nu$ to all vertices of $\mathcal{T}_{u}$ (see Figure~\ref{grid}). We use $\mathcal{F}_{p}$ to denote the set of all the dual edges generated by this subdivision process. For each triangle
$\tau\in \mathcal{T}_h$, we let $h_\tau$ be the diameter of
$\tau$, $h_e$ be the length of edge $e\subset\partial \tau$, and $h=\max\{h_\tau, \tau\in \mathcal{T}_h\}$.
In addition, we define $\mathcal{F}:=\mathcal{F}_{u}\cup \mathcal{F}_{p}$ and $\mathcal{F}^{0}:=\mathcal{F}_{u}^{0}\cup \mathcal{F}_{p}$. We rename the primal element by $S(\nu)$, which is assumed to be star-shaped with respect to a ball of radius $\rho_B h_{S(\nu)}$, where $\rho_B$ is a positive constant. In addition, we assume that for every edge $e\in \partial S(\nu)$, it satisfies $h_e\geq \rho_Eh_{S(\nu)}$, where $\rho_E$ is a positive constant.
 More discussions about mesh regularity assumptions for general meshes can be referred to \cite{Beir13}.
The construction for general meshes is illustrated in Figure~\ref{grid},
where the black solid lines are edges in $\mathcal{F}_{u}$
and the red dotted lines are edges in $\mathcal{F}_{p}$.

Finally, we construct the dual mesh. For each interior edge $e\in \mathcal{F}_{u}^0$, we use $D(e)$ to denote the dual mesh, which is the union of the two triangles in $\mathcal{T}_h$ sharing the edge $e$,
and for each boundary edge $e\in\mathcal{F}_{u}\backslash\mathcal{F}_{u}^0$, we use $D(e)$ to denote the triangle in $\mathcal{T}_h$ having the edge $e$,
see Figure~\ref{grid}.

For each edge $e$, we define
a unit normal vector $\bm{n}_{e}$ as follows: If $e\in \mathcal{F}\setminus \mathcal{F}^{0}$, then
$\bm{n}_{e}$ is the unit normal vector of $e$ pointing towards the outside of $\Omega$. If $e\in \mathcal{F}^{0}$, an
interior edge, we then fix $\bm{n}_{e}$ as one of the two possible unit normal vectors on $e$.
When there is no ambiguity,
we use $\bm{n}$ instead of $\bm{n}_{e}$ to simplify the notation. In addition, we use $\bm{t}$ to denote the corresponding unit tangent vector. We also introduce some notations that will be employed throughout this paper. Let $k\geq 0$ be the order of approximation. For every $\tau
\in \mathcal{T}_{h}$ and $e\in\mathcal{F}$,
we define $P^{k}(\tau)$ and $P^{k}(e)$ as the spaces of polynomials of degree less than or equal to $k$ on $\tau$ and $e$, respectively. In the following, we use $\nabla_h$ and $\text{div}_h$ to denote the element-wise gradient and divergence operators, respectively.


We now define jump terms which will be used throughout the paper.
For each triangle $\tau_i$ in $\mathcal{T}_h$ such that $e \subset \partial \tau_i$, we let $\bm{n}_i$ be the outward unit normal vector on $e\subset \partial \tau_i$. The sign $\delta_i$ of $\bm{n}_i$ with respect to $\bs{n}$ on $e$ is then given by
\begin{equation*}
\delta_i = \bs{n}_i \cdot \bm{n} = \begin{cases}
1 & \text{ if } \bs{n}_i = \bm{n} \text{ on } e,\\
-1 & \text{ if } \bs{n}_i = -\bm{n} \text{ on } e.
\end{cases}
\end{equation*}
For a double-valued scalar quantity $\phi$, let $\phi_i=\phi\mid_{\tau_i}$.The jump $[\phi]$ across an edge $e \in \mathcal{F}^0$ can then be defined as:
\begin{equation*}
[\phi] = \delta_1 \phi_1 + \delta_2 \phi_2,
\end{equation*}
where $\tau_1$ and $\tau_2$ are the two triangles sharing the common edge $e$. For $e\in \mathcal{F}\backslash\mathcal{F}^0$, we let $[\phi] =\phi_1$.
Similarly, for a vector quantity $\bs{\phi}$ and a matrix quantity $\matr{\Phi}$, we let $\bs{\phi}_i=\bs{\phi}\mid_{\tau_i}$ and $\matr{\Phi}_i=\matr{\Phi}\mid_{\tau_i}$, then the jumps $[\bs{\phi} \cdot \bs{n}]$ and $[\matr{\Phi} \bs{n}]$ across an edge $e \in \mathcal{F}^0$ are defined as:
\begin{equation}
\begin{split}
[\bs{\phi} \cdot \bs{n}] & = \delta_1 (\bs{\phi}_1 \cdot \bs{n}) + \delta_2 (\bs{\phi}_2 \cdot \bs{n}),\\
[\matr{\Phi} \bs{n}] & = \delta_1 (\matr{\Phi}_1 \bs{n}) + \delta_2 (\matr{\Phi}_2 \bs{n}).
\label{eq:jump2}
\end{split}
\end{equation}
In addition, for $e\in \mathcal{F}\backslash \mathcal{F}^0$, we define $[\bs{\phi} \cdot \bs{n}]=\bs{\phi}_1 \cdot \bs{n}$ and $[\matr{\Phi} \bs{n}]=\matr{\Phi}_1 \bs{n}$. In the sequel, we use $C$ to denote a positive constant
 which may have different values at different occurrences.

\begin{figure}[t]
\centering
\includegraphics[width=0.7\textwidth]{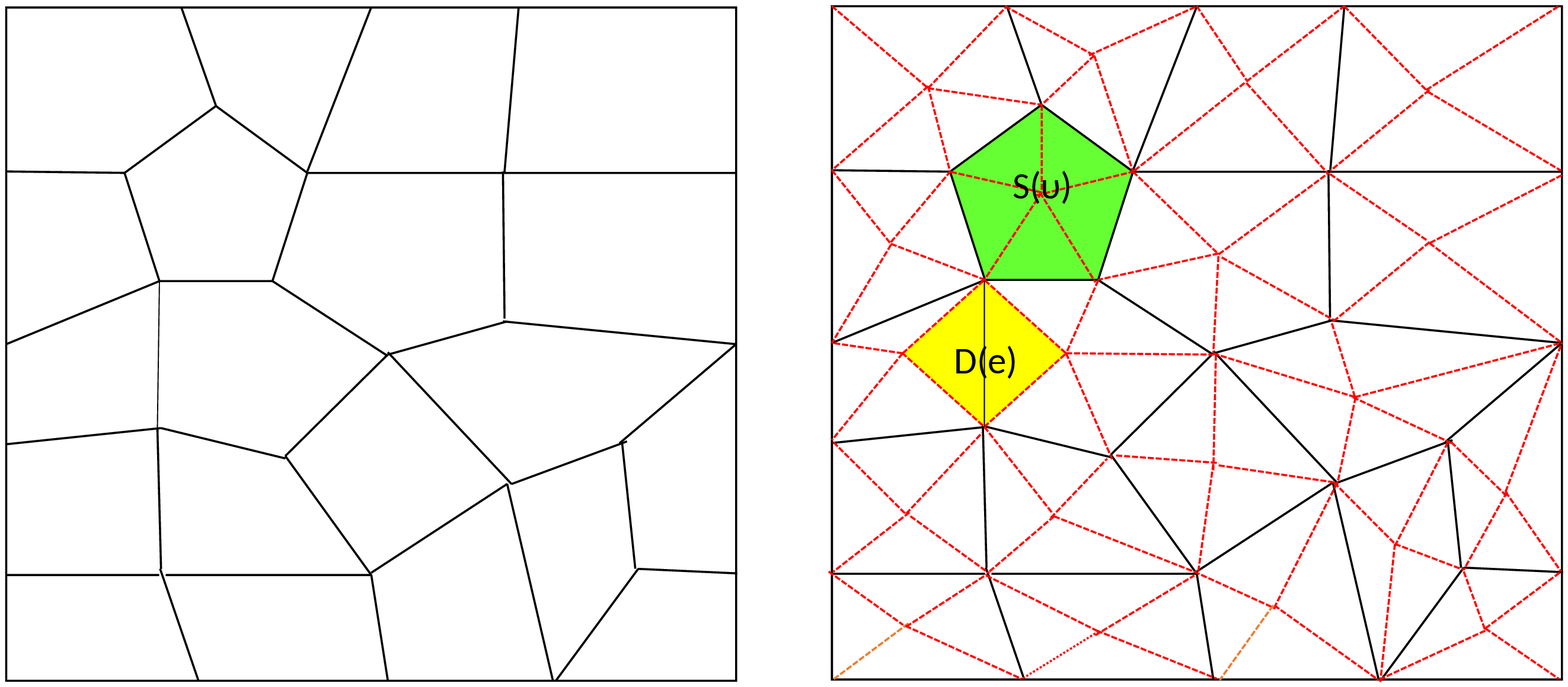}
\setlength{\abovecaptionskip}{-0.5cm}
\caption{Schematic of the primal mesh $S(\nu)$, the dual mesh $D(e)$ and the primal simplicial submeshes.}
\label{grid}
\end{figure}

Now we are ready to introduce the finite dimensional spaces for our staggered DG method
by following \cite{LinaChungLam20}.
We first define
the following finite element space for velocity:
\begin{equation*}
U^h=\{\bs{v} \: : \:  \bs{v} \vert_{\tau} \in P^k(\tau)^2;  \tau \in \mathcal{T}_h; \bs{v} \cdot \bs{n} \text{ is continuous over } e \in \mathcal{F}_p\}.
\end{equation*}
In this space, we define
\begin{align*}
\|\bm{v}\|_{Z_2}^3&=\|\nabla_h \bm{v}\|_{L^3(\Omega)}^3
+\sum_{e\in \mathcal{F}_u}h_e^{-2}\|[\bm{v}]\|_{L^3(e)}^3+
\sum_{e\in \mathcal{F}_p}h_e^{-2}\|[(\bm{v}\cdot \bm{t})\bm{t}]\|_{L^3(e)}^3.
\end{align*}
Next, we define the following finite element space for velocity gradient:
\begin{equation*}
W^h=\{\matr{G} \: : \:  \matr{G} \vert_{\tau} \in P^k(\tau)^{2\times 2};  \tau \in \mathcal{T}_h; \matr{G} \bs{n} \text{ is continuous over } e\in \mathcal{F}_u^0\}.
\end{equation*}
Then, we define the following locally $H^{1}(\Omega)$-conforming finite element space for pressure:
\begin{equation*}
P^h=\{q \: : \:  q \vert_{\tau} \in P^k(\tau);  \tau \in \mathcal{T}_h; q \text{ is continuous over } e\in \mathcal{F}_u^0;\int_{\Omega} q \;dx=0\},
\end{equation*}
which is equipped by
\begin{align*}
\|q\|_{3/2,h}^{3/2}&=\sum_{\tau\in \mathcal{T}_{h}}\int_{\tau} |\nabla q |^{3/2}\;dx
+\sum_{e\in \mathcal{F}_{p}}h_e^{-1/2}\int_e|[q]|^{3/2}\;ds.
\end{align*}
Finally, we define the following space which is employed to
enforce the weak continuity of the velocity gradient over the dual edge
\begin{equation*}
\widehat{U}^h=\{\widehat{\bs{v}} \: : \:  \widehat{\bs{v}} \vert _e \in P^k(e)^2; \widehat{\bs{v}} \cdot \bs{n}\mid_e = 0 \quad\forall e \in \mathcal{F}_p \}.
\end{equation*}

Following \cite{LinaChungLam20}, we can formulate our staggered DG formulation for the unsteady Darcy-Forchheimer-Brinkman problem \eqref{sys0}: Find $(L_h,\bm{u}_h,\widehat{\bm{u}}_h,p_h)\in W^h\times U^h\times \widehat{U}^h\times P^h$ such that
\begin{equation}
\begin{split}
(L_h,G)&=\sqrt{\epsilon}B_h^*(\bm{u}_h,G)+\sqrt{\epsilon}T_h^*(\widehat{\bm{u}}_h,G)\quad \forall G\in W^h,\\ (\partial_t \bm{u}_h,\bm{v})+\sqrt{\epsilon}B_h(L_h,\bm{v})+(\mathcal{N}(\bm{u}_h),\bm{v})
+b_h^*(p_h,\bm{v})&=(\bm{f},\bm{v})\quad \forall \bm{v}\in U^h,\\
-b_h(\bm{u}_h,q)&=0\quad \forall q\in P^h,\\
T_h(L_h,\widehat{\bm{v}})&=0\quad \forall \widehat{\bm{v}}\in \widehat{U}^h,
\end{split}
\label{eq:SDG-discrete}
\end{equation}
where the bilinear forms are defined by
\begin{align*}
B_h^*(\bm{v},G)&=-\int_\Omega \bm{v}\cdot\text{div}_h \matr{G}\;dx+\sum_{e\in \mathcal{F}_p}\int_e(\bm{v}\cdot \bm{n})\bm{n}\cdot [G\bm{n}]\;ds,\\
B_h(G,\bm{v})&=\int_\Omega G\cdot \nabla_h \bm{v}\;dx-\sum_{e\in \mathcal{F}_u}\int_e[\bm{v}]\cdot( G\bm{n})\;ds-\sum_{e\in \mathcal{F}_p}\int_e [(\bm{v}\cdot \bm{t})\bm{t}\cdot (G\bm{n})]\;ds,\\
T_h^*(\widehat{\bm{v}},G)&=\sum_{e\in \mathcal{F}_p}\int_e \widehat{\bm{v}}\cdot [G\bm{n}]\;ds,\\
T_h(G,\widehat{\bm{v}})&=\sum_{e\in \mathcal{F}_p}\int_e [G\bm{n}]\cdot \widehat{\bm{v}}\;ds,\\
b_h^*(q,\bm{v})&=-\int_\Omega q\,\text{div}_h \bm{v}\;dx+\sum_{e\in \mathcal{F}_u}\int_e q[\bm{v}\cdot\bm{n}]\;ds,\\
b_h(\bm{v},q)&=\int_\Omega \bm{v}\cdot\nabla_h q\;dx-\sum_{e\in \mathcal{F}_p}\int_e\bm{v}\cdot\bm{n}[q]\;ds.
\end{align*}
Performing integration by parts reveals the following adjoint properties
\begin{equation}
\begin{split}
B_h(G,\bm{v})&=B_h^*(\bm{v},G)\quad\, \forall (G,\bm{v})\in W^h\times U^h,\\
b_h(\bm{v},q)&=b_h^*(q,\bm{v})\qquad \forall (\bm{v},q)\in U^h\times P^h,\\
T_h(G,\widehat{\bm{v}})&=T_h^*(\widehat{\bm{v}},G)\hspace{0.6cm} \forall (G,\widehat{\bm{v}})\in W^h\times \widehat{U}^h.
\end{split}
\label{eq:adjoint}
\end{equation}

To facilitate the analysis, we define the subspace of $W^h$ by
\begin{align*}
\widehat{W}^h:=\{G\in W^h: \int_e [G\bm{n}] \cdot \hat{\bm{v}}\;ds=0\quad \forall \hat{\bm{v}}\in \widehat{U}^h, \forall e\in\mathcal{F}_p\}.
\end{align*}

Based on the definition of $\widehat{W}^h$ and the discrete formulation \eqref{eq:SDG-discrete}, we can conclude that $L_h\in \widehat{W}^h$. Therefore, we can reformulate our discrete formulation \eqref{eq:SDG-discrete} and obtain the following equivalent formulation: Find $(L_h(t),\bm{u}_h(t),p_h(t))\in \widehat{W}^h\times U^h\times P^h$ such that
\begin{equation}
\begin{split}
(L_h,G)&=\sqrt{\epsilon}B_h^*(\bm{u}_h,G), \\
(\partial_t \bm{u}_h,\bm{v})+\sqrt{\epsilon}B_h(L_h,\bm{v})
+b_h^*(p_h,\bm{v})+(\mathcal{N}(\bm{u}_h),\bm{v})&=(\bm{f},\bm{v}),\\
-b_h(\bm{u}_h,q)&=0
\end{split}
\label{eq:SDG-discrete-new}
\end{equation}
for all $(G,\bm{v},q)\in \widehat{W}^h\times U^h\times P^h$.

For later analysis, we state the following inf-sup condition
\begin{align}
\inf_{q\in P^h}\sup_{\bm{v}\in U^h}\frac{b_h(\bm{v},q)}{\|\bm{v}\|_{L^3(\Omega)}\normmm{q}_{3/2,h}}\geq C
\label{eq:infsupbh}
\end{align}
and
\begin{align}
\inf_{\bm{v}\in U^h}\sup_{G\in \widehat{W}^h}\frac{B_h(G,\bm{v})}{\|G\|_{L^{3/2}(\Omega)}\|\bm{v}\|_{Z_2}}\geq C.
\label{eq:inf-supB}
\end{align}
We remark that the proof of the above inf-sup conditions can follow the techniques
employed in \cite{LinaChungLam20,LinaChungParkZhou19} and we omit it for simplicity.

The inf-sup condition \eqref{eq:inf-supB} implies the existence of an interpolation operator
 $\Pi_h:H^1(\Omega)^{2\times2}\rightarrow \widehat{W}^h$ such that
\begin{align}
B_h(L-\Pi_hL,\bm{v})=0\quad \forall \bm{v}\in U^h.\label{eq:BhPi}
\end{align}
By the standard theory for polynomial preserving operators
 (cf. \cite{Ciarlet78,ChungWave2}), we obtain
\begin{align}
\|L-\Pi_hL\|_{L^2(\Omega)}&\leq C h^{k+1}|L|_{H^{k+1}(\Omega)}.\label{eq:Lerror}
\end{align}
To facilitate later analysis, we also need to define the following projection operator.
Let $I_h: H^1(\Omega)\rightarrow P^h$ be defined by
\begin{align*}
	(I_h q-q,\phi)_e
		&=0 \quad \forall \phi\in P^k(e),\forall e\in \mathcal{F}_{u},\\
	(I_hq-q,\phi)_\tau
		&=0\quad \forall \phi\in P^{k-1}(\tau),\forall \tau\in \mathcal{T}_h
\end{align*}
and let $J_h: H^1(\Omega)^2\rightarrow U^h$ be defined by
\begin{align*}
	((J_h\bm{v}-\bm{v})\cdot\bm{n},\varphi)_e
		&=0\quad \forall \varphi\in P^{k}(e),\ \forall e\in \mathcal{F}_{p},\\
	(J_h\bm{v}-\bm{v}, \bm{\phi})_\tau
		&=0\quad \forall \bm{\phi}\in P^{k-1}(\tau)^2,\ \forall \tau\in \mathcal{T}_h.
\end{align*}
It is easy to see that $I_h$ and $J_h$ are well defined polynomial preserving operators.
In addition, the following approximation properties hold for $\bm{v}\in H^{k+1}(\Omega)^2$ (cf. \cite{Ciarlet78,ChungWave2})
\begin{align}
\|\bm{v}-J_h\bm{v}\|_{L^2(\Omega)}&\leq C h^{k+1}|\bm{v}|_{H^{k+1}(\Omega)}\label{eq:uerror},\\
\|\bm{v}-J_h\bm{v}\|_{L^4(\Omega)}&\leq C h^{k+1}|\bm{v}|_{W^{k+1,4}(\Omega)}\label{eq:uerror4}.
\end{align}
Furthermore, the definitions of $I_h$ and $J_h$ imply directly that
\begin{align}
b_h(\bm{v}-J_h\bm{v},q)&=0\quad \forall q\in P^h, \label{eq:bhJ}\\
b_h^*(q-I_hq,\bm{v})&=0\quad \forall \bm{v}\in U^h. \label{eq:bhI}
\end{align}

\section{Error analysis}\label{sec:analysis}
In this section, we first perform error analysis and obtain rates of convergence for
the semi-discrete scheme. Then we introduce the fully discrete scheme
by using backward Euler for the time discretization, and a error analysis is established
for the resulting fully discrete scheme.

\subsection{Error analysis for semi-discrete scheme}
In this subsection we establish the unique solvability and convergence estimates
 for the semi-discrete scheme.
To this end, we first introduce the following lemma,
which states the unique solvability and stability.

\begin{lemma}
There exists a unique solution to \eqref{eq:SDG-discrete-new}, in addition, the following estimate holds
\begin{align}
\int_0^t(\|L_h\|_{L^2(\Omega)}^2+\alpha_{\textnormal{min}}\|\bm{u}_h\|_{L^2(\Omega)}^2)\;ds
+\frac{1}{2}\|\bm{u}_h(t)\|_{L^2(\Omega)}^2\leq C\int_0^t\|\bm{f}\|_{L^2(\Omega)}^2\;ds.\label{eq:stability2}
\end{align}

\end{lemma}

\begin{proof}

We can proceed
similarly to Theorem~\ref{thm:continuous} to infer that there exists a unique
solution to \eqref{eq:SDG-discrete-new}.

Now we show the stability estimate \eqref{eq:stability2}.
Taking $\bm{v}=\bm{u}_h$, $G=L_h$ and $q=p_h$ in \eqref{eq:SDG-discrete-new},
and summing up the resulting equations yields
\begin{align}
\|L_h\|_{L^2(\Omega)}^2+(\partial_t\bm{u}_h,\bm{u}_h)+(\mathcal{N}(\bm{u}_h),\bm{u}_h)\leq \|\bm{f}\|_{L^2(\Omega)}\|\bm{u}_h\|_{L^2(\Omega)}.\label{eq:stab}
\end{align}
According to the definition of $\mathcal{N}$, we can easily see that
\begin{align*}
(\mathcal{N}(\bm{u}_h),\bm{u}_h)\geq \alpha_{\text{min}}\|\bm{u}_h\|_{L^2(\Omega)}^2.
\end{align*}
Thereby we can infer from \eqref{eq:stab} that
\begin{align*}
\|L_h\|_{L^2(\Omega)}^2+\frac{1}{2}\frac{d}{dt}\|\bm{u}_h\|_{L^2(\Omega)}^2+
\alpha_{\text{min}}\|\bm{u}_h\|_{L^2(\Omega)}^2\leq \frac{1}{2\alpha_{\text{min}}}\|\bm{f}\|_{L^2(\Omega)}^2
+\frac{\alpha_{\text{min}}}{2}\|\bm{u}_h\|_{L^2(\Omega)}^2.
\end{align*}
Integrating over time and using the fact that $\bm{u}_h(0)=0$ yield
\begin{align*}
\int_0^t(\|L_h\|_{L^2(\Omega)}^2+\alpha_{\text{min}}\|\bm{u}_h\|_{L^2(\Omega)}^2)
+\frac{1}{2}\|\bm{u}_h(t)\|_{L^2(\Omega)}^2\leq C\int_0^t\|\bm{f}\|_{L^2(\Omega)}^2.
\end{align*}
Therefore, the proof is completed.
%
%

\end{proof}

\begin{lemma}\label{lemma:energy}
Let $(L,\bm{u},p)$ be the weak solution of \eqref{eq:continuous}
and $(L_h,\bm{u}_h,p_h)$ be the numerical solution of \eqref{eq:SDG-discrete-new},
then the following identity holds
\begin{equation}
\begin{split}
&\frac{1}{2}\|(J_h\bm{u}-\bm{u}_h)(t)\|_{L^2(\Omega)}^2+\int_0^t\Big(\|\Pi_hL-L_h\|_{L^2(\Omega)}^2
+(\mathcal{N}(J_h\bm{u})-\mathcal{N}(\bm{u}_h),J_h\bm{u}-\bm{u}_h)\Big)\;ds\\
&\;=\int_0^t\Big((J_h\bm{u}_t-\bm{u}_t,J_h\bm{u}-\bm{u}_h)
+(\Pi_hL-L, \Pi_hL-L_h)
+(\mathcal{N}(J_h\bm{u})-\mathcal{N}(\bm{u}),J_h\bm{u}-\bm{u}_h)\Big)\;ds.
\label{eq:identity}
\end{split}
\end{equation}

\end{lemma}

\begin{proof}

Replacing $L_h,\bm{u}_h,p_h$ by $L,\bm{u},p$ in \eqref{eq:SDG-discrete-new} yields the following error equations
%
\begin{equation}
\begin{split}
(L-L_h,G)&=\sqrt{\epsilon}B_h^*(\bm{u}-\bm{u}_h,G),\\
(\partial_t(\bm{u}-\bm{u}_h),\bm{v})+\sqrt{\epsilon}B_h(L-L_h,\bm{v})
+b_h^*(p-p_h,\bm{v})
+(\mathcal{N}(\bm{u})-\mathcal{N}(\bm{u}_h),\bm{v})
&=0,\\
b_h(\bm{u}-\bm{u}_h,q)&=0
\end{split}
\label{eq:error}
\end{equation}
for all $(G,\bm{v},q)\in \widehat{W}^h\times U^h\times P^h$.

Taking $G=\Pi_hL-L_h$, $\bm{v}=J_h\bm{u}-\bm{u}_h$ and $q=I_hp-p_h$ in \eqref{eq:error} and adding the resulting equations, then we can infer from \eqref{eq:BhPi},
\eqref{eq:bhJ} and \eqref{eq:bhI} that
\begin{align*}
(\partial_t(\bm{u}-\bm{u}_h),J_h\bm{u}-\bm{u}_h)+(L-L_h,\Pi_hL-L_h)
+(\mathcal{N}(\bm{u})-\mathcal{N}(\bm{u}_h),J_h\bm{u}-\bm{u}_h)=0,
\end{align*}
which can be rewritten as
\begin{align*}
&\frac{1}{2}\frac{d}{dt}\|J_h\bm{u}-\bm{u}_h\|_{L^2(\Omega)}^2+\|\Pi_hL-L_h\|_{L^2(\Omega)}^2+
(\mathcal{N}(J_h\bm{u})-\mathcal{N}(\bm{u}_h),J_h\bm{u}-\bm{u}_h)\\
&\;=(J_h\bm{u}_t-\bm{u}_t,J_h\bm{u}-\bm{u}_h)
+(\Pi_hL-L, \Pi_hL-L_h)
+(\mathcal{N}(J_h\bm{u})-\mathcal{N}(\bm{u}),J_h\bm{u}-\bm{u}_h).
\end{align*}
Integrating over time leads to the desired estimate.

\end{proof}

\begin{theorem}\label{thm:u}
Let $(L,\bm{u},p)$ be the weak solution of \eqref{eq:continuous}
and $(L_h,\bm{u}_h,p_h)$ be the numerical solution of \eqref{eq:SDG-discrete-new}.
Assume that $L\in L^2(0,T;H^{k+1}(\Omega)^{2\times 2})$,
$\bm{u}\in C(0,T;H^{k+1}(\Omega)^{2})\cap L^2(0,T;W^{k+1,4}(\Omega)^{2})$
and $\bm{u}_t\in L^2(0,T;H^{k+1}(\Omega)^{2})$, then we have
\begin{align*}
&\|(\bm{u}-\bm{u}_h)(t)\|_{L^2(\Omega)}^2+\int_0^t\Big(\|L-L_h\|_{L^2(\Omega)}^2+
 \alpha_{\textnormal{min}}\|\bm{u}-\bm{u}_h\|_{L^2(\Omega)}^2\Big)\;ds\\
&\;\leq C\Big( h^{2(k+1)}\|\bm{u}\|_{C(0,T;H^{k+1}(\Omega))}^2
+\int_0^t h^{2(k+1)}\Big(\|\bm{u}\|_{W^{k+1,4}(\Omega)}^2+\|L\|_{H^{k+1}(\Omega)}^2
+\|\bm{u}_t\|_{H^{k+1}(\Omega)}^2\Big)\;ds\Big).
\end{align*}
%

\end{theorem}

\begin{proof}
The proof is based on the estimation of the right hand side of \eqref{eq:identity}.
The Cauchy-Schwarz inequality yields
\begin{align*}
(J_h\bm{u}_t-\bm{u}_t,J_h\bm{u}-\bm{u}_h)&\leq
\|J_h\bm{u}_t-\bm{u}_t\|_{L^2(\Omega)}\|J_h\bm{u}-\bm{u}_h\|_{L^2(\Omega)}\\
&\leq \frac{1}{2\alpha_{\text{min}}}\|J_h\bm{u}_t-\bm{u}_t\|_{L^2(\Omega)}^2
+\frac{\alpha_{\text{min}}}{2}\|J_h\bm{u}-\bm{u}_h\|_{L^2(\Omega)}^2,\\
(\Pi_hL-L, \Pi_hL-L_h)&\leq \|\Pi_hL-L\|_{L^2(\Omega)}\|\Pi_hL-L_h\|_{L^2(\Omega)}\\
&\leq \frac{1}{2} \|\Pi_hL-L\|_{L^2(\Omega)}^2+\frac{1}{2}\|\Pi_hL-L_h\|_{L^2(\Omega)}^2.
\end{align*}
It follows from \eqref{eq:property1} and \eqref{eq:property2} that
\begin{align*}
(\mathcal{N}(J_h\bm{u})-\mathcal{N}(\bm{u}),J_h\bm{u}-\bm{u}_h)&\leq \alpha_{\text{max}} \|J_h\bm{u}-\bm{u}\|_{L^2(\Omega)}\|J_h\bm{u}-\bm{u}_h\|_{L^2(\Omega)}\\
&\;+\beta_{\text{max}}\|J_h\bm{u}-\bm{u}\|_{L^4(\Omega)}
(\|J_h\bm{u}\|_{L^4(\Omega)}+\|\bm{u}\|_{L^4(\Omega)})\|J_h\bm{u}-\bm{u}_h\|_{L^2(\Omega)}
\end{align*}
and
\begin{align*}
(\mathcal{N}(J_h\bm{u})-\mathcal{N}(\bm{u}_h),J_h\bm{u}-\bm{u}_h)\geq \alpha_{\text{min}}\|J_h\bm{u}-\bm{u}_h\|_{L^2(\Omega)}^2.
\end{align*}
Thus we can infer from Lemma~\ref{lemma:energy} that
\begin{align*}
\frac{1}{2}\|(J_h\bm{u}-\bm{u}_h)(t)\|_{L^2(\Omega)}^2+\int_0^t\Big(\|\Pi_hL-L_h\|_{L^2(\Omega)}^2
+\alpha_{\text{min}} \|J_h\bm{u}-\bm{u}_h\|_{L^2(\Omega)}^2\Big)\;ds\leq C\int_0^t\mathcal{M}\;ds,
\end{align*}
where
\begin{align*}
\mathcal{M}=\alpha_{\text{max}}\|J_h\bm{u}-\bm{u}\|_{L^2(\Omega)}^2
+\beta_{\text{max}}\|J_h\bm{u}-\bm{u}\|_{L^4(\Omega)}^2
+\|\Pi_hL-L\|_{L^2(\Omega)}^2+\|J_h\bm{u}_t-\bm{u}_t\|_{L^2(\Omega)}^2.
\end{align*}
An application of the interpolation error estimates \eqref{eq:Lerror}, \eqref{eq:uerror}
 and \eqref{eq:uerror4} leads to the desired estimate.

\end{proof}

\subsection{Error analysis for the fully discrete scheme}

In this subsection we analyze the convergence estimates for the fully discrete scheme. To this end we introduce a partition of the time interval $[0,T]$ into
 subintervals $[t_{n-1},t_n],1\leq n\leq N (N\;\mbox{is an integer})$ and denote the time step size by $\Delta t=\frac{T}{N}$. Using backward Euler scheme in time, we get the fully discrete staggered DG method as follows:
Find $(L_h^{n},\bm{u}_h^n,p_h^n)\in \widehat{W}^h\times U^h\times P^h$ such that
\begin{equation}
\begin{split}
(L_h^{n},G)&=\sqrt{\epsilon}B_h^*(\bm{u}_h^{n},G), \\
(\frac{\bm{u}_h^{n}-\bm{u}_h^{n-1}}{\Delta t},\bm{v})+\sqrt{\epsilon}B_h(L_h^{n},\bm{v})
+b_h^*(p_h^{n},\bm{v})+(\mathcal{N}(\bm{u}_h^{n}),\bm{v})&=(\bm{f}^{n},\bm{v}),\\
-b_h(\bm{u}_h^{n},q)&=0
\end{split}
\label{eq:SDG-fully-discrete}
\end{equation}
for all $(G,\bm{v},q)\in \widehat{W}^h\times U^h\times P^h$.


\begin{lemma}
For any $1\leq n\leq N$, we have the following estimate
\begin{align*}
&\sum_{j=1}^n\Big(\|L_h^{j}\|_{L^2(\Omega)}^2+\frac{1}{2}\|\bm{u}_h^j-\bm{u}_h^{j-1}\|_{L^2(\Omega)}^2+
\frac{\alpha_{\textnormal{min}}}{2}\|\bm{u}_h^{j}\|_{L^2(\Omega)}^2\Big)+\frac{1}{2}\|\bm{u}_h^{n}\|_{L^2(\Omega)}^2\\
&\;\leq C(\|\bm{u}_h^0\|_{L^2(\Omega)}^2+\sum_{j=1}^n\|\bm{f}^{j}\|_{L^2(\Omega)}^2).
\end{align*}

\end{lemma}

\begin{proof}
Taking $\bm{v}=\bm{u}_h^{n}$, $G=L_h^{n}, q=p_h^{n}$ in \eqref{eq:SDG-fully-discrete}, and summing up the resulting equations lead to
\begin{align*}
\|L_h^n\|_{L^2(\Omega)}^2+(\frac{\bm{u}_h^{n}-\bm{u}_h^{n-1}}{\Delta t},\bm{u}_h^{n})
+(\mathcal{N}(\bm{u}_h^{n}),\bm{u}_h^{n})=(\bm{f}^{n},\bm{u}_h^{n}).
\end{align*}
%
%
So using the identity $(a-b,a)= \frac{1}{2}(|a|^2-|b|^2+|a-b|^2)$ yields
\begin{align*}
&\|L_h^{n}\|_{L^2(\Omega)}^2+\frac{1}{2}(\|\bm{u}_h^{n}\|_{L^2(\Omega)}^2-\|\bm{u}_h^{n-1}\|_{L^2(\Omega)}^2
+\|\bm{u}_h^n-\bm{u}_h^{n-1}\|_{L^2(\Omega)}^2)
+\alpha_{\text{min}}\|\bm{u}_h^{n}\|_{L^2(\Omega)}^2\\
&\;\leq \frac{1}{2\alpha_{\textnormal{min}}} \|\bm{f}^{n}\|_{L^2(\Omega)}^2+\frac{\alpha_{\textnormal{min}}}{2}\|\bm{u}_h^{n}\|_{L^2(\Omega)}^2.
\end{align*}
Changing $n$ to $j$ and make a summation for $j=1,\cdots,n$ yields
\begin{align*}
&\sum_{j=1}^n\Big(\|L_h^{j}\|_{L^2(\Omega)}^2+\frac{1}{2}\|\bm{u}_h^j-\bm{u}_h^{j-1}\|_{L^2(\Omega)}^2
+\frac{\alpha_{\text{min}}}{2}\|\bm{u}_h^{j}\|_{L^2(\Omega)}^2\Big)
+\frac{1}{2}(\|\bm{u}_h^{n}\|_{L^2(\Omega)}^2-\|\bm{u}_h^{0}\|_{L^2(\Omega)}^2)\\
&\;\leq \sum_{j=1}^n\frac{1}{2\alpha_{\text{min}}}\|\bm{f}^{j}\|_{L^2(\Omega)}^2.
\end{align*}
Therefore, the proof is completed.

\end{proof}

\begin{theorem}
Let $\{(L_h^n,\bm{u}_h^n,p_h^n)\}_{n=1}^N$ be the numerical solutions of \eqref{eq:SDG-fully-discrete}.
Under the assumptions of Theorem~\ref{thm:u} and $\bm{u}_{tt}\in L^2(0,T;L^2(\Omega)^2)$, we have
\begin{align*}
&2\Delta t\sum_{j=1}^n\|L^{j}-L_h^{j}\|_{L^2(\Omega)}^2
+\|\bm{u}^{n}-\bm{u}_h^{n}\|_{L^2(\Omega)}^2+\sum_{j=1}^n
\|\bm{u}^{j}-\bm{u}_h^{j}-(\bm{u}^{j-1}-\bm{u}_h^{j-1})\|_{L^2(\Omega)}^2\\
&\;+2\Delta t\sum_{j=1}^n\alpha_{\textnormal{min}} \|J_h\bm{u}^{j}-\bm{u}_h^{j}\|_{L^2(\Omega)}^2\\
&\leq C\Big(h^{2(k+1)}\|\bm{u}\|_{C(0,T;H^{k+1}(\Omega))}^2
+(\Delta t)^2\int_0^{t_n}\|\bm{u}_{tt}\|_{L^2(\Omega)}^2\;ds\\
&\hspace{5cm}+h^{2(k+1)}\int_0^{t_n}(\|\bm{u}_t\|_{H^{k+1}(\Omega)}^2
+\|\bm{u}\|_{W^{k+1,4}(\Omega)}^2+\|L\|_{H^{k+1}(\Omega)}^2)\;ds\Big).
\end{align*}

\end{theorem}

\begin{proof}

Replacing $L_h,\bm{u}_h,p_h$ by $L,\bm{u},p$ in \eqref{eq:SDG-fully-discrete} yields the following error equations
\begin{equation}
\begin{split}
(L^{n}-L_h^{n},G)&=\sqrt{\epsilon}B_h^*(\bm{u}^{n}-\bm{u}_h^{n},G), \\
(\frac{\bm{u}^{n}-\bm{u}^{n-1}}{\Delta t}-\frac{\bm{u}_h^{n}-\bm{u}_h^{n-1}}{\Delta t},\bm{v})+\sqrt{\epsilon}B_h(L^{n}-L_h^{n},\bm{v})
\qquad&\\
+b_h^*(p^{n}-p_h^{n},\bm{v})+(\mathcal{N}(\bm{u}^{n}),\bm{v})-(\mathcal{N}(\bm{u}_h^{n}),\bm{v})
&=(\frac{\bm{u}^{n}-\bm{u}^{n-1}}{\Delta t}-\bm{u}_t(:,t_n),\bm{v}),\\
-b_h(\bm{u}^{n}-\bm{u}_h^{n},q)&=0
\end{split}
\label{eq:error-full}
\end{equation}
for all $(G,\bm{v},q)\in \widehat{W}^h\times U^h\times P^h$.

Taking $\bm{v}=J_h\bm{u}^{n}-\bm{u}_h^{n}$,  $G=\Pi_hL^{n}-L_h^{n}$ and $q=I_hp^n-p_h^n$  in \eqref{eq:error-full}, we can obtain
\begin{equation}
\begin{split}
&\|\Pi_hL^{n}-L_h^{n}\|_{L^2(\Omega)}^2+\frac{1}{\Delta t}(J_h\bm{u}^{n}-\bm{u}_h^{n}-(J_h\bm{u}^{n-1}-\bm{u}_h^{n-1}),J_h\bm{u}^{n}-\bm{u}_h^{n})\\
&\;+(\mathcal{N}(J_h\bm{u}^{n})-\mathcal{N}(\bm{u}_h^{n}),J_h\bm{u}^{n}-\bm{u}_h^{n})\\
&=(R^n,J_h\bm{u}^{n}-\bm{u}_h^{n})+
(\Pi_hL^{n}-L^{n},\Pi_hL^{n}-L_h^{n})\\
&\quad+\frac{1}{\Delta t}(J_h\bm{u}^{n}-\bm{u}^{n}-(J_h\bm{u}^{n-1}-\bm{u}^{n-1}),J_h\bm{u}^{n}-\bm{u}_h^{n})
+(\mathcal{N}(J_h\bm{u}^{n})-\mathcal{N}(\bm{u}^{n}),J_h\bm{u}^{n}-\bm{u}_h^{n}),
\end{split}
\label{eq:err-full}
\end{equation}
where
\begin{align*}
R^n=\frac{\bm{u}^{n}-\bm{u}^{n-1}}{\Delta t}-\bm{u}_t(:,t_n).
\end{align*}
To bound $R^n$, we use the Taylor's expansion
\begin{align*}
\bm{u}(:,t_n)-\bm{u}(:,t_{n-1})=\Delta t\bm{u}_t(:,t_n)-\int_{t_{n-1}}^{t_n}(t-t_{n-1})\bm{u}_{tt}(:,t)\;ds.
\end{align*}
As a result, we have
\begin{align*}
R^n=-\frac{1}{\Delta t}\int_{t_{n-1}}^{t_n}(t-t_{n-1})\bm{u}_{tt}(:,t)\;ds.
\end{align*}
The Cauchy-Schwarz inequality yields
\begin{align*}
\|R^n\|_{{L^2(\Omega)}}^2\leq \frac{\Delta t}{3}\int_{t_{n-1}}^{t_n}\|\bm{u}_{tt}\|_{L^2(\Omega)}^2\;ds.
\end{align*}
An appeal to \eqref{eq:property1} and \eqref{eq:property2} implies
\begin{align*}
(\mathcal{N}(J_h\bm{u}^{n})-\mathcal{N}(\bm{u}_h^{n}),J_h\bm{u}^{n}-\bm{u}_h^{n})\geq \alpha_{\text{min}} \|J_h\bm{u}^{n}-\bm{u}_h^{n}\|_{L^2(\Omega)}^2
\end{align*}
and
\begin{align*}
\hspace{-0.1cm}(\mathcal{N}(J_h\bm{u}^{n})-\mathcal{N}(\bm{u}^{n}),J_h\bm{u}^{n}-\bm{u}_h^{n})&\leq \alpha_{\text{max}}\|J_h\bm{u}^{n}-\bm{u}^{n}\|_{L^2(\Omega)}\|J_h\bm{u}^{n}-\bm{u}_h^{n}\|_{L^2(\Omega)}+\\
&\;\beta_{\text{max}}\|J_h\bm{u}^{n}-\bm{u}^{n}\|_{L^4(\Omega)}(\|\bm{u}^{n}\|_{L^4(\Omega)}
+\|J_h\bm{u}^{n}\|_{L^4(\Omega)})\|J_h\bm{u}^{n}-\bm{u}_h^{n}\|_{L^2(\Omega)}.
\end{align*}
Thereby, we can infer from \eqref{eq:err-full} and the equality
$(a-b,a)=\frac{1}{2}(|a|^2-|b|^2+|a-b|^2)$ that
\begin{align*}
&\|\Pi_hL^{n}-L_h^{n}\|_{L^2(\Omega)}^2+\frac{1}{2\Delta t}(\|J_h\bm{u}^{n}-\bm{u}_h^{n}\|_{L^2(\Omega)}^2-\|J_h\bm{u}^{n-1}-\bm{u}_h^{n-1}\|_{L^2(\Omega)}^2\\
&\;+\|J_h\bm{u}^{n}-\bm{u}_h^{n}-(J_h\bm{u}^{n-1}-\bm{u}_h^{n-1})\|_{L^2(\Omega)}^2)
+\alpha_{\text{min}} \|J_h\bm{u}^{n}-\bm{u}_h^{n}\|_{L^2(\Omega)}^2\\
&\leq \|R^n\|_{L^2(\Omega)}\|J_h\bm{u}^{n}-\bm{u}_h^{n}\|_{L^2(\Omega)}
+\|\Pi_hL^{n}-L^{n}\|_{L^2(\Omega)}\|\Pi_hL^{n}-L_h^{n}\|_{L^2(\Omega)}\\
&\;+\Big(\frac{1}{\Delta t}\|J_h\bm{u}^{n}-\bm{u}^{n}-(J_h\bm{u}^{n-1}-\bm{u}^{n-1})\|_{L^2(\Omega)}
+\alpha_{\text{max}}\|J_h\bm{u}^{n}-\bm{u}^{n}\|_{L^2(\Omega)}\\
&\hspace{4cm}+\beta_{\text{max}}\|J_h\bm{u}^{n}-\bm{u}^{n}\|_{L^4(\Omega)}(\|\bm{u}^{n}\|_{L^4(\Omega)}
+\|J_h\bm{u}^{n}\|_{L^4(\Omega)})\Big)\|J_h\bm{u}^{n}-\bm{u}_h^{n}\|_{L^2(\Omega)}.
\end{align*}
Young's inequality yields
\begin{align*}
&\|\Pi_hL^{n}-L_h^{n}\|_{L^2(\Omega)}^2+\frac{1}{2\Delta t}\Big(\|J_h\bm{u}^{n}-\bm{u}_h^{n}\|_{L^2(\Omega)}^2-\|J_h\bm{u}^{n-1}-\bm{u}_h^{n-1}\|_{L^2(\Omega)}^2\\
&\;+\|J_h\bm{u}^{n}-\bm{u}_h^{n}-(J_h\bm{u}^{n-1}-\bm{u}_h^{n-1})\|_{L^2(\Omega)}^2\Big)
+\alpha_{\text{min}} \|J_h\bm{u}^{n}-\bm{u}_h^{n}\|_{L^2(\Omega)}^2\\
&\leq C\Big(\|R^n\|_{L^2(\Omega)}^2
+\|\Pi_hL^{n}-L^{n}\|_{L^2(\Omega)}^2
+\frac{1}{\Delta t}\|J_h\bm{u}^{n}-\bm{u}^{n}-(J_h\bm{u}^{n-1}-\bm{u}^{n-1})\|_{L^2(\Omega)}^2\\
&\hspace{5cm}+\|J_h\bm{u}^{n}-\bm{u}^{n}\|_{L^2(\Omega)}^2
+\|J_h\bm{u}^{n}-\bm{u}^{n}\|_{L^4(\Omega)}^2\Big).\\
\end{align*}
Changing $n$ to $j$, multiplying $2\Delta t$ on both sides, making a summation for $j=1$ to $n$, and using the fact that $J_h\bm{u}^0-\bm{u}_h^0=0$, we obtain
\begin{align*}
&2\Delta t\sum_{j=1}^n\|\Pi_hL^{j}-L_h^{j}\|_{L^2(\Omega)}^2+\|J_h\bm{u}^{n}-\bm{u}_h^{n}\|_{L^2(\Omega)}^2
+\sum_{j=1}^n
\|J_h\bm{u}^{j}-\bm{u}_h^{j}-(J_h\bm{u}^{j-1}-\bm{u}_h^{j-1})\|_{L^2(\Omega)}^2\\
&\;+2\Delta t\sum_{j=1}^n\alpha_{\text{min}} \|J_h\bm{u}^{j}-\bm{u}_h^{j}\|_{L^2(\Omega)}^2\\
&\leq C\Big(\Delta t\sum_{j=1}^n\Big(\|R^j\|_{L^2(\Omega)}^2+\|\Pi_hL^{j}-L^{j}\|_{L^2(\Omega)}^2
+\frac{1}{\Delta t}\|J_h\bm{u}^{j}-\bm{u}^{j}-(J_h\bm{u}^{j-1}-\bm{u}^{j-1})\|_{L^2(\Omega)}^2\\
&\hspace{7cm}+\|J_h\bm{u}^{j}-\bm{u}^{j}\|_{L^2(\Omega)}^2
+\|J_h\bm{u}^{j}-\bm{u}^{j}\|_{L^4(\Omega)}^2\Big)\Big).
\end{align*}
Now we estimate the right hand side. The Cauchy-Schwarz inequality and
 the interpolation error estimate \eqref{eq:uerror} yield
\begin{align*}
\frac{1}{\Delta t}\|J_h\bm{u}^{j}-\bm{u}^{j}-(J_h\bm{u}^{j-1}-\bm{u}^{j-1})\|_{L^2(\Omega)}^2&=\frac{1}{\Delta t}\|\int_{t_{j-1}}^{t_j}(\bm{u}_t-J_h\bm{u}_t)\;ds\|_{L^2(\Omega)}^2\\
&\leq C\frac{1}{\Delta t}\int_{t_{j-1}}^{t_j}h^{2(k+1)}|\bm{u}_t|_{H^{k+1}(\Omega)}^2\;ds.
\end{align*}
In addition, we also have from \eqref{eq:Lerror}, \eqref{eq:uerror} and \eqref{eq:uerror4} that
\begin{align*}
\|J_h\bm{u}^{j}-\bm{u}^{j}\|_{L^2(\Omega)}^2&\leq C h^{2(k+1)}\|\bm{u}^j\|_{H^{k+1}(\Omega)}^2,\\
\|J_h\bm{u}^{j}-\bm{u}^{j}\|_{L^4(\Omega)}^2&\leq C h^{2(k+1)}\|\bm{u}^{j}\|_{W^{k+1,4}(\Omega)}^2,\\
\|\Pi_hL^{j}-L^{j}\|_{L^2(\Omega)}^2&\leq C  h^{2(k+1)}\|L^j\|_{H^{k+1}(\Omega)}^2.
\end{align*}
The preceding arguments lead to
\begin{align*}
&2\Delta t\sum_{j=1}^n\|\Pi_hL^{j}-L_h^{j}\|_{L^2(\Omega)}^2
+\|J_h\bm{u}^{n}-\bm{u}_h^{n}\|_{L^2(\Omega)}^2+\sum_{j=1}^n
\|J_h\bm{u}^{j}-\bm{u}_h^{j}-(J_h\bm{u}^{j-1}-\bm{u}_h^{j-1})\|_{L^2(\Omega)}^2\\
&\;+2\Delta t\sum_{j=1}^n\alpha_{\text{min}} \|J_h\bm{u}^{j}-\bm{u}_h^{j}\|_{L^2(\Omega)}^2\\
&\leq C\Big((\Delta t)^2\int_0^{t_n}\|\bm{u}_{tt}\|_{L^2(\Omega)}^2\;ds
+h^{2(k+1)}\int_0^{t_n}(\|\bm{u}_t\|_{H^{k+1}(\Omega)}^2
+\|\bm{u}\|_{W^{k+1,4}(\Omega)}^2+\|L\|_{H^{k+1}(\Omega)}^2)\;ds\Big).
\end{align*}
Therefore, the proof is completed by using the triangle inequality and
the interpolation error estimates \eqref{eq:Lerror} and \eqref{eq:uerror}.

\end{proof}

\section{Numerical experiments}\label{sec:numerical}

In this section we will present several numerical tests to illustrate the
behavior of the fully discrete scheme. In particular, the robustness of our scheme
with respect to the coefficients will be investigated.
For simplicity, we only perform numerical simulation for $k=1$. Before describing the
numerical results, we present the algorithms that will be used.
At each time step $t_n$ for $1\leq n\leq N$, given an initial guess
$(L_h^{n,(0)},\bm{u}_h^{n,(0)},p_h^{n,(0)})$, Picard's iteration generates the sequences
$(L_h^{n,(m)},\bm{u}_h^{n,(m)},p_h^{n,(m)})$ for $m=1,2,3,\cdots$ by using the sequences of linear problems:
\begin{equation*}
\begin{split}
(L_h^{{n},(m)},G)&=\sqrt{\epsilon}B_h^*(\bm{u}_h^{{n},(m)},G), \\
(\frac{\bm{u}_h^{{n},(m)}-\bm{u}_h^{n-1}}{\Delta t},\bm{v})+\sqrt{\epsilon}B_h(L_h^{{n},(m)},\bm{v})
+b_h^*(p_h^{{n},(m)},\bm{v})\qquad&\\
+(\alpha\bm{u}_h^{n,(m)},\bm{v})
+(\beta|\bm{u}_h|^{n,(m-1)}\bm{u}_h^{n,(m)},\bm{v})&=(\bm{f}^{n,(m)},\bm{v}),\\
-b_h(\bm{u}_h^{n,(m)},q)&=0
\end{split}
\end{equation*}
for all $(G,\bm{v},q)\in \widehat{W}^h\times U^h\times P^h$.
We remark that we need to choose $\Delta t$ to be small enough so that the time discretization error
will not affect the convergence rates. Thus, for comparison we also employ second order difference for time
discretization:
\begin{equation*}
\begin{split}
(L_h^{{n},(m)},G)&=\sqrt{\epsilon}B_h^*(\bm{u}_h^{{n},(m)},G), \\
(\frac{3\bm{u}_h^{{n},(m)}-4\bm{u}_h^{n-1}+\bm{u}_h^{n-2}}{2\Delta t},\bm{v})+
\sqrt{\epsilon}B_h(L_h^{{n},(m)},\bm{v})
+b_h^*(p_h^{{n},(m)},\bm{v})\qquad&\\
+(\alpha\bm{u}_h^{n,(m)},\bm{v})
+(\beta|\bm{u}_h|^{n,(m-1)}\bm{u}_h^{n,(m)},\bm{v})&=(\bm{f}^{n,(m)},\bm{v}),\\
-b_h(\bm{u}_h^{n,(m)},q)&=0
\end{split}
\end{equation*}
for all $(G,\bm{v},q)\in \widehat{W}^h\times U^h\times P^h$. In this case,
we can exploit much large time step size without destroying the convergence rates.

For our simulations, we consider the exact solution given by
\begin{align*}\bm{u}=
\left(
  \begin{array}{c}
    \pi x^2(1-x)^2\sin(2\pi y)\sin(2\pi t) \\
    -2 x(1-x)(1-2x)\sin(\pi y)^2\sin(2\pi t)\\
  \end{array}
\right)
\end{align*}
and
\begin{align*}
p=(\sin(x)\cos(y)+\sin(1)(\cos(1)-1))\cos(2\pi t).
\end{align*}
We show the numerical results on square grids and our undisplayed numerical experiments
indicate that our method can be flexibly applied to general polygonal grids, we only display the
results on square grids for the sake of simplicity.
We will investigate the influence of the coefficients for our method.
For this purpose, we fix $\alpha=1$ and choose different values for $\epsilon$ and $\beta$,
and the numerical results at the final time $T=0.1$ are reported in
Table~\ref{table1}-Table~\ref{table2}. We can observe that optimal convergence rates
for velocity and pressure can be obtained for various values of $\epsilon$, and
the convergence rates for velocity gradient deteriorates when $\epsilon$ approaches zero,
which correlates with our previous results in \cite{LinaChungLam20}. In addition, the accuracy of
$L^2$ error of velocity remains almost the same for various values of $\epsilon$. On the other hand,
we can observe that optimal convergence rates can be obtained for various values of $\beta$
and the value of $L^2$ error of velocity is almost the same for various values of $\beta$.
Next, we show the numerical results by using second order difference for the time discretization
 in Table~\ref{table3}-Table~\ref{table4} with much large time step size, and similar performances
 can be observed.

\begin{table}[t]
\begin{center}
{\footnotesize
\begin{tabular}{ccc||c c|c c|c c}
\hline
Brinkman coefficient& Mesh &Time intervals& \multicolumn{2}{|c|}{$\|\bs{u}-\bs{u}_h\|_{0}$} & \multicolumn{2}{|c|}{$\|L-L_h\|_{0}$} & \multicolumn{2}{|c}{$\|p-p_h\|_{0}$}\\
\hline
$\epsilon$ &$h^{-1}$ &$N$ & Error & Order & Error & Order &  Error & Order \\
\hline
$1$&  2 &4  & 2.45e-2 &   N/A    &1.35e-01 &  N/A &6.63e-02 &   N/A \\
 &4 &16 & 6.13e-3 &   1.99   &5.65e-02 & 1.25 &2.37e-02 & 1.48  \\
 &8 &64 & 1.54e-3 &   1.99   &1.51e-02 & 1.89 &5.83e-03 & 2.02 \\
 &16&256& 3.85e-4 &   2.00   &3.90e-03 & 1.96 &1.35e-03 & 2.11 \\
 \hline
$10^{-2}$& 2 &4  & 2.22e-2 &   N/A    &1.45e-02 &  N/A &6.07e-03 &   N/A \\
 &4 &16 & 5.62e-3 &   1.98   &5.27e-03 & 1.46 &1.29e-03 & 2.23  \\
 &8 &64 & 1.48e-3 &   1.91   &1.42e-03 & 1.89 &2.68e-04 & 2.27 \\
 &16&256& 3.82e-4 &   1.95   &3.75e-04 & 1.92 &6.22e-05 & 2.11 \\
 \hline
$10^{-4}$& 2 &4  & 2.22e-2 &   N/A    &1.51e-03 &  N/A &5.93e-03 &   N/A \\
 &4 &16 & 5.62e-3 &   1.98   &6.04e-04 & 1.31 &1.26e-03 & 2.23  \\
 &8 &64 & 1.46e-3 &   1.94   &2.24e-04 & 1.43 &2.59e-04 & 2.28 \\
 &16&256& 3.69e-4 &   1.98   &9.63e-05 & 1.21 &6.05e-05 & 2.10 \\
 \hline
 $10^{-8}$& 2 &4  & 2.22e-2 &   N/A    &1.51e-05 &  N/A &5.93e-03 &   N/A \\
& 4 &16 & 5.62e-3 &   1.98   &6.05e-06 & 1.31 &1.26e-03 & 2.23  \\
& 8 &64 & 1.46e-3 &   1.94   &2.26e-06 & 1.42 &2.59e-04 & 2.28 \\
& 16&256& 3.69e-4 &   1.98   &9.98e-07 & 1.18 &6.05e-05 & 2.10 \\
\hline
\end{tabular}}
\caption{Backward Euler for time discretization: convergence history for $\alpha=1$ and $\beta=1$.}
\label{table1}
\end{center}
\end{table}

\begin{table}[t]
\begin{center}
{\footnotesize
\begin{tabular}{ccc||c c|c c|c c}
\hline
Forchheimer coefficient& Mesh &Time intervals& \multicolumn{2}{|c|}{$\|\bs{u}-\bs{u}_h\|_{0}$} & \multicolumn{2}{|c|}{$\|L-L_h\|_{0}$} & \multicolumn{2}{|c}{$\|p-p_h\|_{0}$}\\
\hline
$\beta$ &$h^{-1}$ &$N$ & Error & Order & Error & Order &  Error & Order \\
\hline
$1$&  2 &4  & 2.45e-2 &   N/A    &1.35e-01 &  N/A &6.63e-02 &   N/A \\
 &4 &16 & 6.13e-3 &   1.99   &5.65e-02 & 1.25 &2.37e-02 & 1.48  \\
 &8 &64 & 1.54e-3 &   1.99   &1.51e-02 & 1.89 &5.83e-03 & 2.02 \\
 &16&256& 3.85e-4 &   2.00   &3.90e-03 & 1.96 &1.35e-03 & 2.11 \\
 \hline
 $10^2$&2 &4  & 2.41e-2 &   N/A    &1.33e-01 &  N/A &6.97e-02 &   N/A \\
 &4 &16 & 6.10e-3 &   1.98   &5.56e-02 & 1.26 &2.33e-02 & 1.57  \\
 &8 &64 & 1.54e-3 &   1.98   &1.51e-02 & 1.88 &5.82e-03 & 2.00 \\
 &16&256& 3.85e-4 &   2.00   &3.89e-03 & 1.95 &1.35e-03 & 2.11 \\
 \hline
 $10^3$& 2 &4  & 2.27e-2 &   N/A    &1.42e-01 &  N/A &1.12e-01 &   N/A \\
 &4 &16 & 5.92e-3 &   1.93   &5.15e-02 & 1.46 &2.23e-02 & 2.32  \\
 &8 &64 & 1.53e-3 &   1.94   &1.46e-02 & 1.82 &5.85e-03 & 1.93 \\
 &16&256& 3.89e-4 &   1.98   &3.87e-03 & 1.91 &1.37e-03 & 2.10 \\
 \hline
$10^4$&2 &4  & 2.34e-2 &   N/A    &1.88e-01 &  N/A &4.75e-01 &   N/A \\
 &4 &16 & 5.84e-3 &   2.00   &6.64e-02 & 1.51 &8.76e-02 & 2.43  \\
 &8 &64 & 1.50e-3 &   1.96   &1.57e-02 & 2.08 &1.19e-02 & 2.88 \\
 &16&256& 3.89e-4 &   1.94   &3.88e-03 & 2.02 &2.15e-03 & 2.47 \\
\hline
\end{tabular}}
\caption{Backward Euler for time discretization: convergence history for $\alpha=1$ and $\epsilon = 1$.}
\label{table2}
\end{center}
\end{table}

\begin{table}[t]
\begin{center}
{\footnotesize
\begin{tabular}{ccc||c c|c c|c c}
\hline
Brinkman coefficient& Mesh &Time intervals& \multicolumn{2}{|c|}{$\|\bs{u}-\bs{u}_h\|_{0}$} & \multicolumn{2}{|c|}{$\|L-L_h\|_{0}$} & \multicolumn{2}{|c}{$\|p-p_h\|_{0}$}\\
\hline
$\epsilon$ &$h^{-1}$ &$N$ & Error & Order & Error & Order &  Error & Order \\
\hline
 1&2 &2  & 2.51e-2 &   N/A    &1.46e-01 &  N/A &6.05e-02 &   N/A \\
& 4 &4 & 6.17e-3 &   2.02   &5.79e-02 & 1.34 &2.45e-02 & 1.30  \\
& 8 &8 & 1.54e-3 &   2.00   &1.52e-02 & 1.93 &5.87e-03 & 2.06 \\
& 16&16& 3.85e-4 &   2.00   &3.91e-03 & 1.96 &1.35e-03 & 2.12 \\
\hline
$10^{-2}$ &2&2  & 2.22e-2 &   N/A &1.39e-02 &  N/A &6.12e-03 &   N/A \\
& 4 &4 & 5.64e-3 &   1.98   &5.29e-03 &1.40 &1.26e-03 & 2.27  \\
& 8 &8 & 1.49e-3 &   1.92   &1.49e-03 & 1.83 &2.69e-04 & 2.23 \\
& 16&16& 3.82e-4 &   1.96   &3.85e-04 & 1.95 &6.22e-05 & 2.11 \\
\hline
$10^{-4}$&2 &2  & 2.22e-2 &   N/A &1.44e-03 &  N/A &5.98e-03 &   N/A \\
& 4 &4 & 5.59e-3 &   1.98   &5.88e-04 &1.29 &1.22e-03 & 2.28  \\
& 8 &8 & 1.46e-3 &   1.94   &2.22e-04 & 1.41 &2.59e-04 & 2.24 \\
& 16&16& 3.69e-4 &   1.98   &9.61e-05 & 1.21 &6.06e-05 & 2.09 \\
\hline
$10^{-8}$ &2&2  &2.22e-2 &   N/A &1.44e-05 &  N/A &5.98e-03 &   N/A \\
& 4 &4 & 5.59e-3 &   1.98   &5.88e-06 &1.29 &1.22e-03 & 2.29  \\
& 8 &8 & 1.46e-3 &   1.94   &2.23e-06 & 1.40 &2.59e-04 & 2.24 \\
& 16&16& 3.69e-4 &   1.98   &9.95e-07 & 1.17 &6.05e-05 & 2.09 \\
\hline
\end{tabular}}
\caption{Second order difference for time discretization: convergence history for $\alpha=1$ and $\beta = 1$.}
\label{table3}
\end{center}
\end{table}

\clearpage

\begin{table}[t]
\begin{center}
{\footnotesize
\begin{tabular}{ccc||c c|c c|c c}
\hline
Forchheimer coefficient& Mesh &Time intervals& \multicolumn{2}{|c|}{$\|\bs{u}-\bs{u}_h\|_{0}$} & \multicolumn{2}{|c|}{$\|L-L_h\|_{0}$} & \multicolumn{2}{|c}{$\|p-p_h\|_{0}$}\\
\hline
$\beta$ &$h^{-1}$ &$N$ & Error & Order & Error & Order &  Error & Order \\
\hline
 1&2 &2  & 2.51e-2 &   N/A    &1.46e-01 &  N/A &6.05e-02 &   N/A \\
& 4 &4 & 6.17e-3 &   2.02   &5.79e-02 & 1.34 &2.45e-02 & 1.30  \\
& 8 &8 & 1.54e-3 &   2.00   &1.52e-02 & 1.93 &5.87e-03 & 2.06 \\
& 16&16& 3.85e-4 &   2.00   &3.91e-03 & 1.96 &1.35e-03 & 2.12 \\
\hline
$10^{2}$ &2&2  & 2.48e-2 &   N/A &1.44e-01 &  N/A &6.19e-02 &   N/A \\
& 4 &4 & 6.14e-3 &   2.01   &5.73e-02 &1.33 &2.45e-02 & 1.34  \\
& 8 &8 & 1.54e-3 &   1.99   &1.52e-02 & 1.92 &5.89e-03 & 2.06 \\
& 16&16& 3.85e-4 &   2.00   &3.90e-03 & 1.96 &1.35e-03 & 2.12 \\
\hline
$10^{3}$&2 &2  & 2.36e-2 &   N/A &1.40e-01 &  N/A &8.64e-02 &   N/A \\
& 4 &4 & 5.98e-3 &   1.97   &5.38e-02 &1.38 &2.64e-02 & 1.70  \\
& 8 &8 & 1.53e-3 &   1.96   &1.48e-02 &1.86&6.13e-03 & 2.11 \\
& 16&16& 3.86e-4 &   1.99   &3.89e-03 & 1.93 &1.38e-03 & 2.15 \\
\hline
$10^{4}$ &2&2  & 2.29e-2 &   N/A &1.47e-01 &  N/A &4.63e-01 &   N/A \\
& 4 &4 & 5.73e-3 &   2.00   &5.58e-02 &1.40 &1.10e-01 & 2.19  \\
& 8 &8 & 1.49e-3 &   1.94   &1.48e-02 & 1.92 &1.43e-02 & 2.82 \\
& 16&16& 3.82e-4 &   1.97   &3.85e-03 & 1.94 &2.26e-03 & 2.66 \\
\hline
\end{tabular}}
\caption{Second order difference for time discretization: convergence history for $\alpha=1$ and $\epsilon = 1$.}
\label{table4}
\end{center}
\end{table}

\section{Conclusion}

In this paper we have developed and analyzed a uniformly stable staggered DG method for
the unsteady Darcy-Forchheimer-Brinkman problem. The unique solvability of the
discrete formulation is proved, in addition, error analysis for both the semi-discrete and fully
discrete scheme is developed. Several numerical experiments are carried out to confirm the
theoretical findings. The numerical results indicate that our method is robust with respect to the
parameters, in particular, the accuracy of velocity remains almost the same for various values of
parameters.

\section*{Acknowledgments}

The research of Eric Chung is partially supported by the Hong
 Kong RGC General Research Fund (Project numbers 14304719 and 14302018)
 and CUHK Faculty of Science Direct Grant 2019-20.

\end{document}